\documentclass[11pt,bookmarks=false]{article}

\usepackage[colorlinks,citecolor=blue,urlcolor=blue]{hyperref}
\usepackage{amssymb}
\usepackage{amsthm}
\usepackage{amsmath}
\usepackage{booktabs}
\usepackage{mathrsfs}
\usepackage{multirow}
\usepackage{bm}
\usepackage{geometry}
\usepackage{graphics}
\usepackage{float}
\usepackage{epsfig}
\usepackage{subfigure}
\usepackage{enumerate}

\numberwithin{equation}{section}
\newtheorem{thm}{Theorem}[section]
\newtheorem{defn}{Definition}[section]
\newtheorem{coro}[thm]{Corollary}

\def\proclaim#1{\par \bigskip\noindent {\bf #1}\bgroup\it\ }
\def\endproclaim{\egroup\par\bigskip}

\def\pr{\mathbb{P}} 
\def\ep{\mathbb{E}}

\def\text#1{\mbox{\rm #1}}

\begin{document}
\title{\bf Ladder Chains: A Variation of Random Walks}
\date{}
\maketitle
\begin{center}
\vskip -1.5cm {{\sc {\large Chenhe Zhang}\footnote[1]{Department of Mathematics, Zhejiang University, Hangzhou 310027, P.R.China; 3150104161@zju.edu.cn}  and  {\large Xiang Fang}\footnote[2]{Department of Mathematics, Zhejiang University, Hangzhou 310027, P.R.China; 3150103685@zju.edu.cn}}}
\end{center}

\bigskip
\noindent{\bf Abstract.} 
The authors propose a new variation of random walks called ladder chains $L(r,s,p)$. We extend concepts such as ruin probability, hitting time, transience and recurrence of random walks to ladder chain. Take $L(2,2,p)$ for instance, we find the linear difference equations that the ruin probability and the hitting time satisfy. We also prove the recurrence of a critical case $(p=\sqrt{2}-1)$. All approaches of these results can be generalized to solve similar problems for other ladder chains.


\section{Introduction}

Research on random walks has a long history. The random walk problem was first formally proposed by Pearson in 1905 \cite{PRWP}. At the same year, Rayleigh solved it and extended the problem to 2-dimensions \cite{PRWR}. Later on, P\'{o}lya discussed the recurrence of random walks of several dimensions in his paper \cite{Polya}. After that, Erd\H{o}s and R\'{e}nyi initiated the study of random graphs in \cite{RGE,RGG} and greatly advanced the research on graph theory. There is also much research on discrete variations of random walks such as L\'{e}vy flight, random walks on Riemannian manifolds, and random walks on finite groups (see \cite{RWS,RWF}). In the past several decades, the related theoretical work has made a tremendous success in various fields, including but not limited to, physics, psychology, computer science, and solving Laplace's equation (see \cite{DLG,EBRW,AAQ,RWHE}).

However, to the best knowledge of the authors, there is little mature research on the variation of random walks as follows. Suppose $\{\xi_n\}$ is a sequence of independent and identically distributed random variables and $\mathbf{X}=(X_n,n\geqslant 1)$ is a stochastic process, where $X_n$ is a simple function of $(\xi_n,\cdots,\xi_{n-r+1})$ unrelated to $n$. Define $S_n=X_1+\cdots+X_n$. Then it is natural to ask the three questions as below.
\begin{enumerate}[(Q-1)]
\item What is the ruin probability for $\{S_n\}$?
\item How to calculate the mean duration of this case?
\item Can we define the transience and recurrence about $\{S_n\}$ similarly to Markov chains, along with an easy criterion?
\end{enumerate}
Although it is easy to see that $(S_n,n\geqslant 1)$ is a second order Markov chain, we cannot directly apply the properties of higher order Markov chain to answer the above questions. Therefore, we start from a new perspective and solve these kinds of problems using the methods for random walks in \cite{Prob}.

The rest of the paper is organized as follows. In Section 2, we give and prove the linear differential equation that the ruin probability satisfies, using reflection principle. This is followed in Section 3 by proving the existence of the mean duration and the linear difference equation it satisfies. Finally, we compute the ``absorption probability'' for a critical case and derive the recurrence of this case in Section 4.

\section{How likely is the gambler ruin?}

Think about such a question. Suppose two gambler A and B are gambling under the rule as follows. If A wins B,  B gives A a coin; otherwise, A gives B a coin. Specially, if A wins B both this round and previous round, B has to give A an extra coin. We already know that the probability that A wins B is $p>0$ and they have $a$ and $b$ coins respectively. How likely is A before B to ruin within finite rounds? It is easy to find that this problem is a variation of the gambler's ruin problem. To solve it, we first introduce some notations.

Consider the Bernoulli scheme $(\Omega,\mathcal{A},\pr)$, where 
$$\Omega=\{\omega:\omega=(x_1,x_2,\cdots),x_i=\pm1\},\ \mathcal{A}=\{A:A\subseteq\Omega\}.$$ 
Let $\xi_1,\xi_2,\cdots,\xi_n,\cdots\ i.i.d.\sim Bernoulli(p)$ and $q=1-p$. Define
\begin{equation}
X_k=\begin{cases}
-1,&\mathrm{if}\ \xi_k=-1,\\
2,&\mathrm{if}\ \xi_k=\xi_{k-1}=1,\\
1,&\mathrm{otherwise},
\end{cases}
\end{equation}
and $S_k=X_1+\cdots+X_k$. For any $n\in\mathbb{N}$ and starting point $x\in\mathbb{R}$, define $S_n^x=x+S_n$ and stopping time
\begin{equation}
\tau_k^x=\begin{cases}
min\{l:S_l^x\leqslant L\ or\ S_l^x\geqslant U,\ 0\leqslant l\leqslant k \},&\exists l=0,\cdots,k,\ s.t.\ S_l^x \leqslant L\ or\ S_l^x\geqslant U,\\
k,&\mathrm{otherwise},
\end{cases}
\label{e taukx}
\end{equation}
where $L<U$ are two integers. Let 
\begin{equation}
\mathcal{A}_k^x=\bigcup_{l=0}^k\{w:\tau_k^x=l\ \mathrm{and}\ S_l^x \leqslant L\},\quad \alpha_k(x)=\mathbb{P}(\mathcal{A}_k^x),
\end{equation}
and
\begin{equation}
\mathcal{B}_k^x=\bigcup_{l=0}^k\{w:\tau_k^x=l\ \mathrm{and}\ S_l^x \geqslant U\},\quad \beta_k(x)=\mathbb{P}(\mathcal{B}_k^x).
\end{equation}
Obviously we have for any positive integer $k$,
\begin{equation}
\begin{cases}
\alpha_k(L)=\beta_k(U)=1,\\ 
\beta_k(L)=\alpha_k(U)=0
\end{cases}
\label{e boundary conditions}
\end{equation}
as the boundary conditions.
\begin{thm}
For any integer $x\in[L,U]$, there exists $0\leqslant\alpha(x)\leqslant 1$ and $0\leqslant\beta(x)\leqslant 1$ such that
$$\alpha(x)=\lim\limits_{k\rightarrow\infty}\alpha_k(x),\quad \beta(x)=\lim\limits_{k\rightarrow\infty}\beta_k(x).$$ 
\label{t exist}
\end{thm}

\begin{proof}
It is easy to see that $\{\mathcal{A}_k^x\}_{k=1}^{\infty}$ is a sequence of increasing events. Therefore,
$$0\leqslant \alpha_1(x)\leqslant\alpha_2(x)\leqslant\cdots\leqslant\alpha_k(x)\leqslant\cdots\leqslant 1.$$ The monotone convergence theorem promises the convergence of $\{\alpha_k(x)\}$, that is $\lim\limits_{k\rightarrow\infty}\alpha_k(x)$ exists and is in $[0,1]$. Similarly, we know $\lim\limits_{k\rightarrow\infty}\beta_k(x)$ exists and is in $[0,1]$. 
\end{proof}

The following theorem gives the values of $\alpha(x)$ and $\beta(x)$.
\begin{thm}
Suppose $U-L\geqslant 4$, then the $\alpha(\cdot)$ defined above are the unique solution of the linear difference equation
\begin{equation}
\alpha(x)=\begin{cases}0,&x=U,\\
q\alpha(x-1),&x=U-1,\\
pq\alpha(x)+q\alpha(x-1),&x=U-2,\\
pq\alpha(x)+p\alpha(x+2)-pq\alpha(x+1)+q\alpha(x-1),&L+1\leqslant x\leqslant U-3,\\
1,&x=L.
\end{cases}
\label{e ax}
\end{equation}
And the $\beta(\cdot)$ defined above are the unique solution of the linear difference equation
\begin{equation}
\beta(x)=\begin{cases}1,&x=U,\\
p+q\beta(x-1),&x=U-1,\\
p^2+pq\beta(x)+q\beta(x-1),&x=U-2,\\
pq\beta(x)+p\beta(x+2)-pq\beta(x+1)+q\beta(x-1),&L+1\leqslant x\leqslant U-3,\\
0,&x=L.
\end{cases}
\label{e bx}
\end{equation}
\label{t axbx}
\end{thm}
\begin{proof}
We only give the proof of $\beta(\cdot)$. For $k\geqslant 2$ we have 
\begin{equation}
\begin{split}
\beta_k(U-1)&=p\mathbb{P}(\mathcal{B}_k^{U-1}|\xi_1=1)+q\mathbb{P}(\mathcal{B}_k^{U-1}|\xi_1=-1)\\
&=p+q\beta_{k-1}(U-2).
\end{split}
\end{equation}
Similarly, 
\begin{equation}
\begin{split}
\beta_k(U-2)&=p\mathbb{P}(\mathcal{B}_{k}^{U-2}|\xi_1=1)+q\mathbb{P}(\mathcal{B}_{k}^{U-2}|\xi_1=-1)\\
 &=p\bigl(\mathbb{P}(\mathcal{B}_{k}^{U-2},\xi_2=1|\xi_1=1)+\mathbb{P}(\mathcal{B}_{k}^{U-2},\xi_2=-1|\xi_1=1)\bigr)+q\beta_{k-1}(U-3).
\end{split}
\end{equation}
Notice that
\begin{equation}
\mathbb{P}(\mathcal{B}_{k}^{U-2},\xi_2=1|\xi_1=1)=\mathbb{P}(\mathcal{B}_{k}^{U-2}|\xi_2=\xi_1=1)\mathbb{P}(\xi_2=1|\xi_1=1)=p
\end{equation}
and
\begin{equation}
\mathbb{P}(\mathcal{B}_{k}^{U-2},\xi_2=-1|\xi_1=1)=\mathbb{P}(\mathcal{B}_{k}^{U-2}|\xi_1=1, \xi_2=-1)\mathbb{P}(\xi_2=-1|\xi_1=1)=q\beta_{k-2}(U-2),
\end{equation}
so we have
\begin{equation}
\beta_k(U-2)=p^2+pq\beta_{k-2}(U-2)+q\beta_{k-1}(U-3).
\end{equation}
For $L+1\leqslant x\leqslant U-3$ and positive integer $k$, define 
\begin{equation}
L_k=\begin{cases}
min\{1\leqslant l\leqslant k:\xi_l=-1\},&\exists l=1,2,\cdots,k,\ s.t.\ \xi_l=-1,\\
k+1,&\forall l=1,\cdots,k,\ \xi_l=1.
\end{cases}
\end{equation}
In other words, $L_k$ is the least positive number $l$ such that $\xi_l=-1$. Then we have
\begin{equation}
\begin{split}
\beta_k(x)&=p\mathbb{P}(\mathcal{B}_{k}^{x}|\xi_1=1)+q\mathbb{P}(\mathcal{B}_{k}^{x}|\xi_1=-1)\\
&=p\bigl(\mathbb{P}\bigl(\mathcal{B}_{k}^{x},L_k=2|\xi_1=1)+\mathbb{P}(\mathcal{B}_{k}^{x},L_k\geqslant 3|\xi_1=1)\bigr)+q\beta_{k-1}(x-1)\\
&=p\bigl(\mathbb{P}\bigl(\mathcal{B}_{k}^{x},L_k=2|\xi_1=1)+p\mathbb{P}(\mathcal{B}_{k}^{x}|L_k\geqslant 3)\bigr)+q\beta_{k-1}(x-1).
\label{e bkx}
\end{split}
\end{equation}

\begin{figure*}[!htb]
\centering
\subfigure[Figure 1: Given $T_k\geqslant 3$] 
{\includegraphics[height=3in,width=3.2in,angle=0]{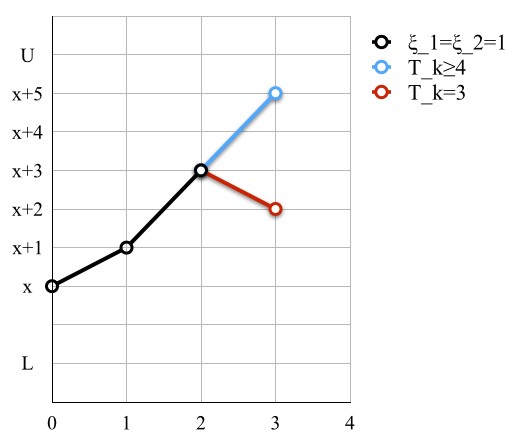}}
\subfigure[Figure 2: Given $\xi_1=1$] {\includegraphics[height=3in,width=2.6in,angle=0]{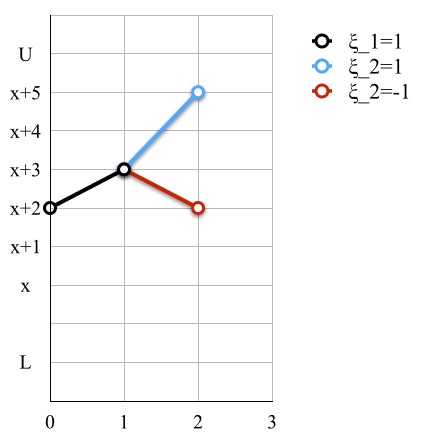}}
\end{figure*}

\noindent From the above two figures we can see that the paths
$$(x,x+1,x+3,x+3+X_3,\cdots,x+3+X_3+\cdots+X_k)$$
in $\mathcal{B}_{k}^{x}$ satisfying $L_k\geqslant 3$ are in one-to-one correspondence with the paths
$$(x+2,x+3,x+3+X_2,\cdots,x+3+X_2+\cdots+X_{k-1})$$
in $\mathcal{B}_{k-1}^{x+2}$ satisfying $\xi_1=1$. By the reflection principle, we know 
\begin{equation}
\pr(\mathcal{B}_{k}^{x}|L_k\geqslant 3)=\pr(\mathcal{B}_{k-1}^{x+2}|\xi_1=1).
\end{equation}
Furthermore, 
\begin{equation}
\beta_{k-1}(x+2)=p\mathbb{P}(\mathcal{B}_{k-1}^{x+2}|\xi_1=1)+q\beta_{k-2}(x+1),
\end{equation}
which means that
\begin{equation}
\pr(\mathcal{B}_{k}^{x}|L_k\geqslant 3)=\frac{\beta_{k-1}(x+2)-q\beta_{k-2}(x+1)}{p}.
\label{e bkxtk3}
\end{equation}
Combine (\ref{e bkx}), (\ref{e bkxtk3}) with
\begin{equation}
\mathbb{P}(\mathcal{B}_{k}^{x},L_k=2|\xi_1=1)=\mathbb{P}(\mathcal{B}_{k}^{x}|\xi_2=-1,\xi_1=1)\mathbb{P}(\xi_2=-1|\xi_1=1)=q\beta_{k-2}(x)
\end{equation}
and we get
\begin{equation}
\beta_k(x)=pq\beta_{k-2}(x)+p\beta_{k-1}(x+2)-pq\beta_{k-2}(x+1)+q\beta_{k-1}(x-1).
\end{equation}
According to Theorem \ref{t exist} and (\ref{e boundary conditions}), we finally find the linear difference equation that $\beta(x)$ satisfies. According to the theory of linear difference equations, (\ref{e bx}) has and only has one solution $\beta(\cdot)$. By the same token, we can get the corresponding conclusion for $\alpha(\cdot)$.
\end{proof}
The answer to the question we asked at first is a direct corollary of the above theorem as follows.
\begin{coro}
The probabilities that A ruins before B within finite rounds and B ruins before A within finite rounds are $\beta(0)$ and $\alpha(0)$ respectively, where $\alpha(\cdot),\beta(\cdot)$ can be computed by (\ref{e ax}) and (\ref{e bx}) after setting $L=-a$ and $U=b$.
\end{coro}
\begin{coro}
For any integer $x\in[L,U]$, $\alpha(x)+\beta(x)=1$, where $U-L\geqslant 4$ are two integers and $\alpha(\cdot),\beta(\cdot)$ follows Theorem \ref{t exist}.
\label{c cx}
\end{coro}
\begin{proof}
Suppose $\gamma(x)=\alpha(x)+\beta(x)-1$, then Theorem \ref{t axbx} gives
\begin{equation}
\gamma(x)=\begin{cases}
0,&x=U,\\
q\gamma(x-1),&x=U-1,\\
pq\gamma(x)+q\gamma(x-1),&x=U-2,\\
pq\gamma(x)+p\gamma(x+2)-pq\gamma(x+1)+q\gamma(x-1),&L+1\leqslant x\leqslant U-3,\\
0,&x=L,\end{cases}
\end{equation}
which has the only solution $\mathbf{0}$. Therefore, $\alpha(x)+\beta(x)=1$.
\end{proof}
According to the Theorem \ref{t axbx} and its proof, whether $U-L\geqslant 4$ or $U-L<4$, $U,L$ and $x$ are integers or not, we always have the following corollary.
\begin{coro}
For any real number $x\in[L,U]$, there exists $0\leqslant\alpha(x)\leqslant 1$ and $0\leqslant\beta(x)\leqslant 1$ such that
$$\alpha(x)=\lim\limits_{k\rightarrow\infty}\alpha_k(x),\quad \beta(x)=\lim\limits_{k\rightarrow\infty}\beta_k(x).$$ 
Furthermore, $\alpha(x)+\beta(x)=1$.
\end{coro}
Hence, as the classic gambler's ruin problem, we can conclude that there must be one of the two gamblers ruin within finite rounds.

\section{When will a gambler ruin?}
We have defined the $\tau_{k}^{x}$ as a stopping time from above. In this section, we will first give the linear difference equation that $m_k(x)=\ep\tau_{k}^{x}$ (obviously it exists because it is bounded) satisfies, and then prove the existence of $\lim\limits_{k\rightarrow\infty}m_k(x)$ for any integer $x\in[L,U]$ and $0<p<1$. For simplicity, we assume that $U$ and $L$ are two integers satisfying $U-L\geqslant 4$ and $x$ is an integer between them in the rest of the paper.
\begin{thm}
The $m_k(\cdot)$ defined above are the unique solution of the linear difference equation
\begin{equation}
m_k(x)=\begin{cases}0,&x=U,\\
1+qm_{k-1}(x-1),&x=U-1,\\
1+p+pqm_{k-2}(x)+qm_{k-1}(x-1),&x=U-2,\\
1+pqm_{k-2}(x)+pm_{k-1}(x+2)-pqm_{k-2}(x+1)+qm_{k-1}(x-1),&L+1\leqslant x\leqslant U-3,\\
0,&x=L.
\end{cases}
\label{e mkx}
\end{equation}
\end{thm}
\begin{proof}
For any given integer $x\in[L,U]$, define $\ep\tau_k^x=m_k(x)$. Proceed as in the derivation of the recurrent relations for $\beta_k(x)$ and we have
\begin{equation}
\begin{split}
m_k(U-1)&=p\ep(\tau_k^{U-1}|\xi_1=1)+q\ep(\tau_k^{U-1}|\xi_1=-1)\\
&=1+qm_{k-1}(U-2),\\
m_k(U-2)&=p\ep(\tau_k^{U-2}|\xi_1=1)+q\ep(\tau_k^{U-2}|\xi_1=-1)\\
&=p\bigl(\ep(\tau_k^{U-2}\mathbf{1}\{L_k=2\}|\xi_1=1)+\ep(\tau_k^{U-2}\mathbf{1}\{L_k\ge 2\}|\xi_1=1)\bigr)+q\ep(\tau_k^{U-2}|\xi_1=-1)\\
&=p\bigl(\ep(\tau_k^{U-2}|L_k=2,\xi_1=1)\mathbb{P}(L_k=2|\xi_1=1)\\
&\quad \ +\ep(\tau_k^{U-2}|L_k\ge 3,\xi_1=1)\mathbb{P}(L_k\ge 3|\xi_1=1)\bigr)+q(1+\ep\tau_{k-1}^{U-3})\\
&=1+p+pqm_{k-2}(U-2)+qm_{k-1}(U-3).
\end{split}
\label{e m1}
\end{equation}
Furthermore,
\begin{equation}
\begin{split}
m_k(x)&=p\ep(\tau_k^x|\xi_1=1)+q\ep(\tau_k^x|\xi_1=-1)\\
&=p\bigl(\ep(\tau_k^x\mathbf{1}\{L_k=2\}|\xi_1=1)+\ep(\tau_k^x\mathbf{1}\{L_k\geqslant 3\}|\xi_1=1)\bigr)+q(1+\ep\tau_{k-1}^{x-1})\\
&=p\bigl(\ep(\tau_k^x|L_k=2,\xi_1=1)\mathbb{P}(L_k=2|\xi_1=1)\\
&\quad \ +\ep(\tau_k^x|L_k\geqslant 3,\xi_1=1)\mathbb{P}(L_k\geqslant 3|\xi_1=1)\bigr)+q(1+\ep\tau_{k-1}^{x-1})\\
&=p\Bigl[q(2+\ep\tau_{k-2}^x)+p\bigl(1+\ep(\tau_{k-1}^{x+2}|\xi_1=1)\bigr)\Bigr]+q+qm_{k-1}(x-1).
\end{split}
\end{equation}
Also,
\begin{equation}
\begin{split}
m_{k-1}(x+2)&=p\ep(\tau_{k-1}^{x+2}|\xi_1=1)+q\ep(\tau_{k-1}^{x+2}|\xi_1=-1)\\
&=p\ep(\tau_{k-1}^{x+2}|\xi_1=1)+q\bigl(1+m_{k-2}(x+1)\bigr).
\end{split}
\end{equation}
Hence, 
\begin{equation}
p\ep(\tau_{k-1}^{x+2}|\xi_1=1)=m_{k-1}(x+2)-q(1+m_{k-2}(x+1)),
\end{equation}
and
\begin{equation}
m_k(x)=1+pqm_{k-2}(x)+pm_{k-1}(x+2)-pqm_{k-2}(x+1)+qm_{k-1}(x-1).
\label{e m2}
\end{equation}
Combining the boundary conditions that 
\begin{equation}
\begin{cases}
m_k(L)=0,\\
m_k(U)=0
\end{cases}
\end{equation}
with (\ref{e m1}) and (\ref{e m2}) gives the linear difference equation showed in (\ref{e mkx}). 

\end{proof}

\begin{thm}
For any integer $x\in[L,U]$ and $0<p<1$, there exists $m(x)$ such that $m(x)=\lim\limits_{k\rightarrow\infty}m_k(x)$. The $m(x)$ is called \emph{mean duration}.
\label{t md}
\end{thm}
\begin{proof}
See Appendix A.
\end{proof}

\section{Ladder Chains}

In this Section, we generalize the $S_n$ in the Section 2 to a specific $r$-th order Markov Chain named \emph{ladder chain}, denoted by $L(r,s,p)$, and prove the recurrence of $L(2,2,\sqrt{2}-1)$.

\subsection{Basic Concepts}
\begin{defn}
Suppose $\xi_1,\xi_2,\cdots,\xi_n,\cdots\ i.i.d.\sim Bernoulli(p)$ and $s$ is a positive integer. Define 
\begin{equation}
X_k=\begin{cases}
-1,&\mathrm{if}\ \xi_k=-1,\\
s,&\mathrm{if}\ \xi_k=\xi_{k-1}=\cdots=\xi_{k-r+1}=1,\\
1,&\mathrm{otherwise},
\end{cases}
\end{equation}
and $S_n=X_1+X_2+\cdots+X_n$. Then we call the stochastic process $\mathbf{S}=(S_n:n\geqslant 1)$ a \emph{ladder chain} with \emph{order} $r$, \emph{step} $s$, and \emph{probability} $p$, simply denoted by $\mathbf{S}\sim L(r,s,p)$.
\end{defn}
Similar to the transience and recurrence of a state of a Markov chain, we can define the transience and recurrence of a ladder chain here.
\begin{defn}
A ladder chain $\mathbf{S}=(S_n:n\geqslant 1)$ is said to be transient if 
\begin{equation}
\pr(S_n\neq 0\mathrm{\ for\ any\ }n\geqslant 1)>0.
\end{equation}
Otherwise, the ladder chain $\mathbf{S}$ is said to be recurrent.
\end{defn}

\subsection{Recurrence of $L(2,2,\sqrt{2}-1)$}
Before demonstrating the recurrence of the ladder chain $L(2,2,\sqrt{2}-1)$, we first prove a rather intuitive theorem about the ruin probability when $U$ or $L$ are infinite.
\begin{thm}
Suppose $\mathbf{S}=(S_{n}:n\geqslant 1)\sim L(2,2,p)$ and x is a given positive integer. Then for $p=\sqrt{2}-1$,
\begin{equation}
\pr(S_n\geqslant x\mathrm{\ for\ some\ }n\geqslant 1)=\pr(S_n\leqslant -x\mathrm{\ for\ some\ }n\geqslant 1)=1.
\end{equation}
\label{t leave}
\end{thm}
\begin{proof}
It is easy to see that the event
\begin{equation}
\begin{split}
\{\omega:S_n\geqslant x\mathrm{\ for\ some\ }n\geqslant 1\}&\Leftrightarrow\{\omega:\mathrm{the\ particle\ leaves\ }(-\infty,x)\mathrm{\ from\ }0\}\\
&\Leftrightarrow\left.\left(\lim\limits_{k\rightarrow\infty}\mathcal{B}_{k}^{0}\right)\right|_{L=-\infty}^{U=x},
\end{split}
\end{equation}
and
\begin{equation}
\begin{split}
\{\omega:S_n\leqslant -x\mathrm{\ for\ some\ }n\geqslant 1\}&\Leftrightarrow\{\omega:\mathrm{the\ particle\ leaves\ }(-x,+\infty)\mathrm{\ from\ }0\}\\
&\Leftrightarrow\left.\left(\lim\limits_{k\rightarrow\infty}\mathcal{A}_{k}^{0}\right)\right|_{L=-x}^{U=+\infty}.
\end{split}
\end{equation}
Therefore, we have
\begin{equation}
\pr(S_n\geqslant x\mathrm{\ for\ some\ }n\geqslant 1)=\beta(0)\Bigl|_{L=-\infty}^{U=x},\quad \pr(S_n\leqslant -x\mathrm{\ for\ some\ }n\geqslant 1)=\alpha(0)\Bigr|_{L=-x}^{U=+\infty}.
\end{equation}
Solving the characteristic equation of (\ref{e bx}) gives
\begin{equation}
\left\{\begin{split}
&\lambda_1=\lambda_2=1,\\
&\lambda_3=-\sqrt{2}.
\end{split}\right.
\end{equation}
Therefore, the general formula of $\beta(x)$ is given by
\begin{equation}
\beta(x)=c_1+c_2x+c_3\lambda_3^x,
\end{equation}
where 
\begin{equation}
\begin{cases}
c_1=\frac{p^2L\lambda_3^{U-1}-pL\lambda_3^{U-2}+pqL\lambda_3^{U-3}+p(q-p)\lambda_3^{L}}{p(q-p)\lambda_3^L+q(1-2q-p(U-L))\lambda_3^{U-3}+(p(U-L)+2q^2-p)\lambda_3^{U-2}+(-p^2(U-L)+2p^2-q)\lambda_3^{U-1}};\\
c_2=\frac{-p^2\lambda_3^{U-1}+p\lambda_3^{U-2}-pq\lambda_3^{U-3}}{p(q-p)\lambda_3^L+q(1-2q-p(U-L))\lambda_3^{U-3}+(p(U-L)+2q^2-p)\lambda_3^{U-2}+(-p^2(U-L)+2p^2-q)\lambda_3^{U-1}};\\
c_3=\frac{p(p-q)}{p(q-p)\lambda_3^L+q(1-2q-p(U-L))\lambda_3^{U-3}+(p(U-L)+2q^2-p)\lambda_3^{U-2}+(-p^2(U-L)+2p^2-q)\lambda_3^{U-1}}.
\end{cases}
\end{equation}
Since for fixed $U\in\mathbb{R}$,
\begin{equation}
\lim\limits_{L\rightarrow-\infty}c_1=1,\quad \lim\limits_{L\rightarrow-\infty}c_2=\lim\limits_{L\rightarrow-\infty}c_3=0,
\end{equation}
and for fixed $L\in\mathbb{R}$,
\begin{equation}
\lim\limits_{U\rightarrow+\infty}c_1=\lim\limits_{U\rightarrow+\infty}c_2=\lim\limits_{U\rightarrow+\infty}c_3=0,
\end{equation}
we know that 
\begin{equation}
\beta(0)\Bigl|_{L=-\infty}^{U=x}=1,\ \beta(0)\Bigr|_{L=-x}^{U=+\infty}=0.
\end{equation}
It suffices to apply $\alpha(0)+\beta(0)=1$ for any $L<U$ and thus we have
\begin{equation}
\alpha(0)\Bigr|_{L=-x}^{U=+\infty}=1,
\end{equation}
which completes the proof.
\end{proof}
It is easy to see that the $p_0=\sqrt{2}-1$ here is the probability such that $\ep X_k=0,\forall\ k\geqslant 2$ of $L(2,2,p_0)$. We end the paper with a theorem about the transience and recurrence of the ladder chain $L(2,2,p)$. Since the corresponding properties of ladder chains with higher orders and steps can be derived by a similar method, we omit them in this paper.
\begin{thm}
The ladder chain $L(2,2,\sqrt{2}-1)$ is recurrent.
\end{thm} 
\begin{proof}
Define $L_k'=\mathrm{inf}\{n\geqslant k:\xi_n=-1\}$, $T=\mathrm{inf}\{n\geqslant 1:S_n=0\}$, and $T'=\mathrm{inf}\{n\geqslant 1:S_n=0\mathrm{\ or\ }1\}$. Obviously, $\forall k\geqslant 1,\ \pr(L_k'<\infty)=1$. On one hand,
\begin{equation}
\begin{split}
\pr(\xi_1=1,T<\infty)&=\pr(\xi_1=1,T<\infty,L_2'<\infty)\\
&=\sum\limits_{k=2}^{\infty}\pr(\xi_1=1,T<\infty,L_2'=k)\\
&=\sum\limits_{k=2}^{\infty}\pr(T<\infty|\xi_1=1,L_2'=k)\pr(\xi_1=1,L_2'=k)\\
&=\sum\limits_{k=2}^{\infty}\pr(S_n\leqslant 0\mathrm{\ for\ some\ }n\geqslant 1|\xi_1=1,L_2'=k)\pr(\xi_1=1,L_2'=k)\\
&=pq+\sum\limits_{k=3}^{\infty}\pr(S_{n-k}\leqslant 2(2-k) \mathrm{\ for\ some\ }n\geqslant k+1)\pr(\xi_1=1,L_2'=k).\\
&=pq+\sum\limits_{k=3}^{\infty}\pr(S_{n}\leqslant 2(2-k) \mathrm{\ for\ some\ }n\geqslant 1)p^{k-1}q.
\end{split}
\end{equation}
By Theorem \ref{t leave}, we find $\pr(S_{n}\leqslant 2(2-k) \mathrm{\ for\ some\ }n\geqslant 1)=1$ for any $k\geqslant 3$ and consequently
\begin{equation}
\pr(\xi_1=1,T<\infty)=\sum\limits_{k=2}^{\infty}p^{k-1}q=p.
\end{equation}
On the other hand,
\begin{equation}
\begin{split}
\pr(\xi_1=-1,T<\infty)&=\pr(\xi_1=-1,T<\infty,T'<\infty)\\
&=\sum\limits_{k=2}^{\infty}\pr(\xi_1=-1,T<\infty,T'=k)\\
&=\sum\limits_{k=2}^{\infty}\pr(\xi_1=-1,T<\infty,T'=k,S_k=0)\\
&\ \ +\sum\limits_{k=2}^{\infty}\pr(\xi_1=-1,T<\infty,T'=k,S_k=1)
\end{split}
\end{equation}
For any $\omega\in\Omega$ satisfying $\{\xi_1=-1,T<\infty,T'=k,S_k=1\}$, assume for contradiction that $\xi_k=-1$; then it follows that $X_k=-1$ and $S_{k-1}=2$, which means that there exists some $l<k-1$ such that $S_l=0$ or 1 and $T'=\mathrm{inf}\{n\geqslant 1:S_n=0\mathrm{\ or\ }1\}\leqslant l<k$.
\begin{equation}
\begin{split}
\pr(\xi_1=-1,T<\infty)&=\sum\limits_{k=2}^{\infty}\pr(\xi_1=-1,T<\infty,T'=k,S_k=0)\\
&\quad \ +\sum\limits_{k=2}^{\infty}\pr(\xi_1=-1,T<\infty,T'=k,S_k=1,\xi_k=1)\\
&=\sum\limits_{k=2}^{\infty}\pr(\xi_1=-1,T'=k,S_k=0)\\
&\quad \ +\sum\limits_{k=2}^{\infty}\sum\limits_{l=k+1}^{\infty}\pr(\xi_1=-1,T<\infty,T'=k,S_k=1,\xi_k=1,L_{k+1}'=l)
\end{split}
\label{e p1}
\end{equation}
For any $l\geqslant k+1$, according to Theorem \ref{t leave},
\begin{equation}
\begin{split}
&\pr(\xi_1=-1,T<\infty,T'=k,S_k=1,\xi_k=1,L_{k+1}'=l)\\
=&\pr(S_n\leqslant 0\mathrm{\ for\ some\ }n\geqslant l+1,\xi_1=-1,T'=k,S_k=1,\xi_k=1,L_{k+1}'=l)\\
=&\pr(S_{n-l}\leqslant 2(k+1-l)\mathrm{\ for\ some\ }n\geqslant l+1|\xi_1=-1,T'=k,S_k=1,\xi_k=1,L_{k+1}'=l)\\
&\cdot\pr(\xi_1=-1,T'=k,S_k=1,\xi_k=1,L_{k+1}'=l)\\
=&\pr(S_{n}\leqslant 2(k+1-l)\mathrm{\ for\ some\ }n\geqslant 1)\pr(\xi_1=-1,T'=k,S_k=1,\xi_k=1,L_{k+1}'=l)\\
=&\pr(\xi_1=-1,T'=k,S_k=1)p^{l-k-1}q.
\end{split}
\label{e p2}
\end{equation}
Substituting (\ref{e p2}) into (\ref{e p1}) and applying Theorem \ref{t leave} again give
\begin{equation}
\begin{split}
\pr(\xi_1=-1,T<\infty)&=\sum\limits_{k=2}^{\infty}\pr(\xi_1=-1,T'=k,S_k=0)\\
&\quad \ +\sum\limits_{k=2}^{\infty}\pr(\xi_1=-1,T'=k,S_k=1)\\
&=\sum\limits_{k=2}^{\infty}\pr(T'=k|\xi_1=-1)\pr(\xi_1=-1)\\
&=\sum\limits_{k=2}^{\infty}\pr(S_n\geqslant 1\mathrm{\ for\ some\ }n\geqslant 1)\pr(\xi_1=-1)=q
\end{split}
\end{equation}
Therefore,
\begin{equation}
\pr(T<\infty)=\pr(\xi_1=1,T<\infty)+\pr(\xi_1=-1,T<\infty)=p+q=1,
\end{equation}
which completes the proof.
\end{proof}
\section*{\centering{Appendix A: Proof of Theorem \ref{t md}}}
\textbf{Theorem 3.2.} \emph{For any integer $x\in[L,U]$ and $0<p<1$, there exists $m(x)$ such that $m(x)=\lim\limits_{k\rightarrow\infty}m_k(x)$. The $m(x)$ is called \emph{mean duration}.}
\begin{proof}
We first prove that $m_k(x)$ is bounded on $[L,U]$ for any $0<p<1$, where $U$ and $L$ are two integers such that $U-L\geqslant 4$. We demonstrate it in two cases.\\
$\bullet$\textbf{(Case a)} $\mathbf{p\neq\sqrt{2}-1}$.\\
The random variable $S_{\tau_n^x}^x=\sum\limits_{k=0}^{n}S_k^x\mathbf{1}\{\tau_n^x=k\}$ describes the walk at the stopping time $\tau_n^x$. Computing the expectation of both sides gives
\begin{equation}
\begin{split}
\mathbb{E}S_{\tau_n^x}^x&=\mathbb{E}\sum\limits_{k=0}^{n}(S_k^x-S_n^x+S_n^x)\mathbf{1}\{\tau_n^x=k\}\\
&=\mathbb{E}S_n^x-\mathbb{E}\sum\limits_{k=0}^{n-1}(X_{k+1}+\cdots+X_n)\mathbf{1}\{\tau_n^x=k\}\\
&=\mathbb{E}S_n^x-\sum\limits_{k=0}^{n-1}(\mathbb{E}X_{k+1}\mathbf{1}\{\tau_n^x=k\}+\mathbb{E}(X_{k+2}+\cdots+X_n)\mathbf{1}\{\tau_n^x=k\}).
\label{e e11}
\end{split}
\end{equation}
From the definition of $X_k$ and $\tau_n^x$, we can conclude that $\mathbf{1}\{\tau_n^x=k\}$ is only related with $(\xi_1,\cdots,\xi_k)$, and $(X_{k+2},\cdots,X_n)$ is only related with $(\xi_{k+1},\xi_{k+2},\cdots,\xi_n)$. Hence, by the independence among $\xi_1,\cdots,\xi_n$ we know that $\mathbf{1}\{\tau_n^x=k\}$ and $(X_{k+2},\cdots,X_n)$ are mutually independent. It follows that
\begin{equation}
\mathbb{E}(X_{k+2}+\cdots+X_n)\mathbf{1}\{\tau_n^x=k\}=(n-k-1)(p^2+2p-1)\mathbb{P}(\tau_n^x=k).
\label{e e12}
\end{equation}
Besides, we have 
\begin{equation}
\mathbb{E}S_n^x=\mathbb{E}(x+X_1+\cdots+X_n)=x+p-q+(n-1)(p^2+2p-1)
\label{e e13}
\end{equation}
and
\begin{equation}
\sum\limits_{k=0}^{n-1}\mathbb{E}X_{k+1}\mathbf{1}\{\tau_n^x=k\}=\sum\limits_{k=0}^{n-1}\mathbb{E}X_{k+1}\mathbf{1}\{\tau_n^x=k,S_k^x\leqslant L\}+\sum\limits_{k=0}^{n-1}\mathbb{E}X_{k+1}\mathbf{1}\{\tau_n^x=k,S_k^x\geqslant U\}.
\label{e e14}
\end{equation}
If $S_{\tau_n^x}^x\leqslant L$, then 
$\begin{cases}\mathbb{P}(X_{k+1}=1)=p\\\mathbb{P}(X_{k+1}=-1)=q\end{cases}$, which shows that 
\begin{equation}
\mathbb{E}X_{k+1}\mathbf{1}\{\tau_n^x=k,S_k^x\leqslant L\}=(p-q)\mathbb{P}(\tau_n^x=k,S_{\tau_n^x}^x\leqslant L).
\label{e e15}
\end{equation}
Similarly, if $S_{\tau_n^x}^x\geqslant U$, then 
$\begin{cases}\mathbb{P}(X_{k+1}=2)=p\\\mathbb{P}(X_{k+1}=-1)=q\end{cases}$, which shows that 
\begin{equation}
\mathbb{E}X_{k+1}\mathbf{1}\{\tau_n^x=k,S_k^x\geqslant U\}=(2p-q)\mathbb{P}(\tau_n^x=k,S_{\tau_n^x}^x\geqslant U).
\label{e e16}
\end{equation}
Substitute (\ref{e e12}), (\ref{e e13}), (\ref{e e14}), (\ref{e e15}), (\ref{e e16}) into (\ref{e e11}) and we find 
\begin{equation}
\begin{split}
\mathbb{E}S_{\tau_n^x}^x&=x+p-q+(n-1)(p^2+2p-1)\\
&\quad \ -(p-q)\sum\limits_{k=0}^{n-1}\mathbb{P}(\tau_n^x=k,S_k^x\leqslant L)-(2p-q)\sum\limits_{k=0}^{n-1}\mathbb{P}(\tau_n^x=k,S_k^x\geqslant U)\\
&\quad \ -(p^2+2p-1)\Bigl[(n-1)\bigl(1-\mathbb{P}(\tau_n^x=n)\bigr)-\sum\limits_{k=0}^{n}k\mathbb{P}(\tau_n^x=k)+n\mathbb{P}(\tau_n^x=n)\Bigr]\\
&=x+p-q-(p^2+2p-1)\mathbb{P}(\tau_n^x=n)-(p-q)(\mathbb{P}(\mathcal{A}_n^x)-\mathbb{P}(\tau_n^x=n,S_n^x\leqslant L))\\
&\quad \ -(2p-q)\Bigl(\mathbb{P}(\mathcal{B}_n^x)-\mathbb{P}(\tau_n^x=n,S_n^x\geqslant U)\Bigr)+(p^2+2p-1)\mathbb{E}\tau_n^x.
\end{split}
\end{equation}
Notice that in this case $p^2+2p-1\neq0$ and $-L\leqslant S_{\tau_n^x}^x\leqslant U\ \mathrm{a.s.}$ Therefore, we have
\begin{equation}
m_n(x)=\mathbb{E}\tau_n^x\leqslant 1+\frac{2max\{|L|,|U|\}+2|p-q|+|2p-q|}{|p^2+2p-1|}
\end{equation}
is bounded on $[L,U]$.\\
$\bullet$\textbf{(Case b)} $\mathbf{p=\sqrt{2}-1}$.\\
\begin{equation}
\begin{split}
\mathbb{E}(S_{\tau_n^x}^x)^2&=\sum\limits_{k=0}^{n}\mathbb{E}(S_{\tau_n^x}^x)^2\mathbf{1}\{\tau_n^x=k\}\\
&=\sum\limits_{k=0}^{n}\mathbb{E}(S_k^x-S_n^x+S_n^x)^2\mathbf{1}\{\tau_n^x=k\}\\
&=\mathbb{E}(S_n^x)^2+\sum\limits_{k=0}^{n-1}\mathbb{E}(S_n^x-S_k^x)^2\mathbf{1}\{\tau_n^x=k\}-2\sum\limits_{k=0}^{n-1}\mathbb{E}S_n^x(S_n^x-S_k^x)\mathbf{1}\{\tau_n^x=k\}\\
&=\mathbb{E}(S_n^x)^2-\sum\limits_{k=0}^{n-1}\mathbb{E}(S_n^x-S_k^x)^2\mathbf{1}\{\tau_n^x=k\}-2\sum\limits_{k=0}^{n-1}\mathbb{E}S_k^x(S_n^x-S_k^x)\mathbf{1}\{\tau_n^x=k\}\\
&=\mathbb{E}(S_n^x)^2-\sum\limits_{k=0}^{n-1}\mathbb{E}(X_{k+1}+\cdots+X_n)^2\mathbf{1}\{\tau_n^x=k\}-2\sum\limits_{k=0}^{n-1}\mathbb{E}S_k^x(X_{k+1}+\cdots+X_n)\mathbf{1}\{\tau_n^x=k\}
\end{split}
\label{e e21}
\end{equation}
In this case, we have $\mathbb{E}X_2=\cdots=\mathbb{E}X_n=0$. Simple computation gives 
\begin{equation}
\begin{cases}
VarX_1=1-(p-q)^2=12p-4,\\
VarX_k=4p^2+pq+q=4-6p,&k\geqslant 2,\\
Cov(X_1,X_2)=2p^2+q^2-2pq=6-14p,\\
Cov(X_k,X_{k+1})=4p^3+2p^2q+q^2-pq^2-pq-2pq^2=3p-1,&k\geqslant 2
\end{cases}
\end{equation}
Then it follows that 
\begin{equation}
\begin{split}
\mathbb{E}(S_n^x)^2&=Var(x+X_1+\cdots+X_n)+(\mathbb{E}S_n^x)^2\\
&=VarX_1+\cdots+VarX_n+2\bigl[Cov(X_1,X_2)+Cov(X_2,X_3)+\cdots+Cov(X_{n-1},X_n)\bigr]+(\mathbb{E}S_n^x)^2\\
&=8+2n-22p+(x+p-q)^2.
\end{split}
\label{e e22}
\end{equation}
Notice that 
\begin{equation}
\begin{split}
&\sum\limits_{k=0}^{n-1}\mathbb{E}S_k^x(X_{k+1}+\cdots+X_n)I\{\tau_n^x=k\}=\sum\limits_{k=0}^{n-1}\mathbb{E}S_k^xX_{k+1}I\{\tau_n^x=k\}\\
=&\sum\limits_{k=0}^{n-1}\mathbb{E}S_k^xX_{k+1}\bigl(I\{\tau_n^x=k,S_k^x\leqslant L\}+I\{\tau_n^x=k,S_k^x\geqslant U\}\bigr)\\
=&\sum\limits_{k=0}^{n-1}\mathbb{E}(S_k^xX_{k+1}|\tau_n^x=k,S_k^x\leqslant L)\mathbb{P}(\tau_n^x=k,S_k^x\leqslant L)\\
&+\sum\limits_{k=0}^{n-1}\mathbb{E}(S_k^xX_{k+1}|\tau_n^x=k,S_k^x\geqslant U)\mathbb{P}(\tau_n^x=k,S_k^x\geqslant U)\\
=&\sum\limits_{k=0}^{n-1}(p-q)\mathbb{E}(S_k|\tau_n^x=k,S_k^x\leqslant L)\mathbb{P}(\tau_n^x=k,S_k^x\leqslant L)\\
&+\sum\limits_{k=0}^{n-1}(2p-q)\mathbb{E}(S_k|\tau_n^x=k,S_k^x\geqslant U)\mathbb{P}(\tau_n^x=k,S_k^x\geqslant U).
\end{split}
\end{equation}
So we have
\begin{equation}
\left|\sum\limits_{k=0}^{n-1}\mathbb{E}S_k^x(X_{k+1}+\cdots+X_n)I\{\tau_n^x=k\}\right|\leqslant (2p-q)max\{|L|,|U|\}. 
\label{e e23}
\end{equation}
Besides, 
\begin{equation}
\begin{split}
\mathbb{E}\sum\limits_{k=0}^{n-1}(X_{k+1}+\cdots+X_n)^2I\{\tau_n^x=k\}&=\mathbb{E}\sum\limits_{k=0}^{n-1}X_{k+1}^2I\{\tau_n^x=k\}+\mathbb{E}\sum\limits_{k=0}^{n-2}(X_{k+2}+\cdots+X_n)^2I\{\tau_n^x=k\}\\
&\quad \ +2\mathbb{E}\sum\limits_{k=0}^{n-2}X_{k+1}(X_{k+2}+\cdots+X_n)I\{\tau_n^x=k\}
\end{split}
\label{e e24}
\end{equation}
First, 
\begin{equation}
\mathbb{E}\sum\limits_{k=0}^{n-1}X_{k+1}^2I\{\tau_n^x=k\}\leqslant \mathbb{E}\sum\limits_{k=0}^{n-1}4I\{\tau_n^x=k\}\ \Rightarrow\ 0\leqslant\mathbb{E}_1\leqslant 4.
\label{e e25}
\end{equation}
Also,
\begin{equation}
\begin{split}
&\mathbb{E}\sum\limits_{k=0}^{n-2}(X_{k+2}+\cdots+X_n)^2I\{\tau_n^x=k\}\\
=&\sum\limits_{k=0}^{n-2}\Bigl[Var(X_{k+2}+\cdots+X_n)+\bigl(\mathbb{E}(X_{k+2}+\cdots+X_n)\bigr)^2\Bigr]\mathbb{P}(\tau_n^x=k)\\
=&\sum\limits_{k=0}^{n-2}(2n-2k-6p)\mathbb{P}(\tau_n^x=k)=2n-6p+(6p-2)\mathbb{P}(\tau_n^x=n-1)+6p\mathbb{P}(\tau_n^x=n)-2\mathbb{E}\tau_n^x\\
\leqslant& 2n-2\mathbb{E}\tau_n^x.
\end{split}
\label{e e26}
\end{equation}
At last,
\begin{equation}
\begin{split}
&\forall k=0,\cdots,n-1,j=k+2,\cdots,n,\ -2\leqslant X_{k+1}X_j \leqslant 4,\ a.s.\\
\Rightarrow &-2\leqslant \mathbb{E}\sum\limits_{k=0}^{n-2}X_{k+1}(X_{k+2}+\cdots+X_n)I\{\tau_n^x=k\}\leqslant 4.
\label{e e27}
\end{split}
\end{equation}
By (\ref{e e21}), (\ref{e e22}),  (\ref{e e23}), (\ref{e e24}), (\ref{e e25}), (\ref{e e26}), (\ref{e e27}), and the fact that $L\leqslant S_{\tau_n^x}^x\leqslant U\ \mathrm{a.s.}$ we eventually find
\begin{equation}
\mathbb{E}\tau_k^x\leqslant 2+11p+(2p-q)\mathrm{max}\{|L|,|U|\}+\frac{(x+p-q)^2+\mathrm{max}\{L^2,U^2\}}{2},
\end{equation}
which means that $\{m_k(x)\}$ is bounded on $[L,U]$. From the denifition of $\tau_k^x$ in (\ref{e taukx}), we can see obviously that $\tau_k^x\leqslant\tau_{k+1}^x,\mathrm{\ a.s.}$, which means that $\{m_k(x)\}$ is non-decreasing for any fixed $x$. Hence, $m(x)=\lim\limits_{k\rightarrow\infty}m_k(x)$ exists for any integer $x\in[L,U]$ according to the monotone convergence theorem.
\end{proof}

\section*{\centering{Acknowledgements}}

We are grateful to Prof. Zhonggen Su and Prof. Qinghai Zhang for many useful discussions and suggestions on systemizing our ideas and polishing this paper.

\end{document}